\documentclass[12pt]{article}

\usepackage{amsmath,amssymb,amsbsy,amsfonts,amsthm,latexsym,
                    amsopn,amstext,amsxtra,euscript,amscd}

\begin{document}

\newtheorem{theorem}{Theorem}
\newtheorem{lemma}[theorem]{Lemma}
\newtheorem{claim}[theorem]{Claim}
\newtheorem{cor}[theorem]{Corollary}
\newtheorem{prop}[theorem]{Proposition}
\newtheorem{definition}{Definition}
\newtheorem{question}[theorem]{Question}

\def\cA{{\mathcal A}}
\def\cB{{\mathcal B}}
\def\cC{{\mathcal C}}
\def\cD{{\mathcal D}}
\def\cE{{\mathcal E}}
\def\cF{{\mathcal F}}
\def\cG{{\mathcal G}}
\def\cH{{\mathcal H}}
\def\cI{{\mathcal I}}
\def\cJ{{\mathcal J}}
\def\cK{{\mathcal K}}
\def\cL{{\mathcal L}}
\def\cM{{\mathcal M}}
\def\cN{{\mathcal N}}
\def\cO{{\mathcal O}}
\def\cP{{\mathcal P}}
\def\cQ{{\mathcal Q}}
\def\cR{{\mathcal R}}
\def\cS{{\mathcal S}}
\def\cT{{\mathcal T}}
\def\cU{{\mathcal U}}
\def\cV{{\mathcal V}}
\def\cW{{\mathcal W}}
\def\cX{{\mathcal X}}
\def\cY{{\mathcal Y}}
\def\cZ{{\mathcal Z}}

\def\A{{\mathbb A}}
\def\B{{\mathbb B}}
\def\C{{\mathbb C}}
\def\D{{\mathbb D}}
\def\E{{\mathbb E}}
\def\F{{\mathbb F}}
\def\G{{\mathbb G}}
\def\I{{\mathbb I}}
\def\J{{\mathbb J}}
\def\K{{\mathbb K}}
\def\L{{\mathbb L}}
\def\M{{\mathbb M}}
\def\N{{\mathbb N}}
\def\O{{\mathbb O}}
\def\P{{\mathbb P}}
\def\Q{{\mathbb Q}}
\def\R{{\mathbb R}}
\def\S{{\mathbb S}}
\def\T{{\mathbb T}}
\def\U{{\mathbb U}}
\def\V{{\mathbb V}}
\def\W{{\mathbb W}}
\def\X{{\mathbb X}}
\def\Y{{\mathbb Y}}
\def\Z{{\mathbb Z}}

\def\ep{{\mathbf{e}}_p}
\def\eq{{\mathbf{e}}_q}
\def\eqf{{\mathbf{e}}_{q_f}}

\def\es{{\mathbf{e}}_s}

\def\scr{\scriptstyle}
\def\\{\cr}
\def\({\left(}
\def\){\right)}
\def\[{\left[}
\def\]{\right]}
\def\<{\langle}
\def\>{\rangle}
\def\fl#1{\left\lfloor#1\right\rfloor}
\def\rf#1{\left\lceil#1\right\rceil}
\def\le{\leqslant}
\def\ge{\geqslant}
\def\eps{\varepsilon}
\def\mand{\qquad\mbox{and}\qquad}

\def\vec#1{\mathbf{#1}}
\def\inv#1{\overline{#1}}
\def\vol#1{\mathrm{vol}\,{#1}}
\def\dist{\mathrm{dist}}

\def\SL{\mathrm{SL}}

\def\Hba{\overline{\cH}_{a,m}}
\def\Hta{\widetilde{\cH}_{a,m}}
\def\Hb1{\overline{\cH}_{m}}
\def\Ht1{\widetilde{\cH}_{m}}

\def\Zm{\Z/m\Z}

\def \tX{\widetilde{X}}
\def\tY{\widetilde{Y}}
\def\tZ{\widetilde{Z}}

\def\Err{{\mathbf{E}}}

\newcommand{\comm}[1]{\marginpar{%
\vskip-\baselineskip 
\raggedright\footnotesize
\itshape\hrule\smallskip#1\par\smallskip\hrule}}

\def\xxx{\vskip5pt\hrule\vskip5pt}


\title{On Small Solutions to Quadratic Congruences}

\author{
{\sc Igor E. Shparlinski} \\
{Department of Computing, Macquarie University} \\
{Sydney, NSW 2109, Australia} \\
{igor@ics.mq.edu.au}}

\date{\today}
\pagenumbering{arabic}

\maketitle

\begin{abstract}
We estimate the deviation of  the number of solutions 
of the congruence
$$
m^2-n^2 \equiv c \pmod q, \qquad 1 \le m \le M, \ 1\le n \le N,
$$ 
from its expected value
on average over $c=1, \ldots, q$. This estimate is motivated by the recently established
by D.~R. Heath-Brown connection between the distibution of solution to
this congruence  
and the pair correlation problem for the fractional 
parts of the quadratic function $\alpha k^2$, $k=1,2,\ldots$ 
with a real $\alpha$. 
\end{abstract}



\section{Introduction}

For positive integers  $M$, $N$ and $q$ and 
an arbitrary integer $c$, we denote 
$$
A(M,N;q,c) = 
\# \{1\le m \le M, \ 1 \le n \le N~:~ m^2-n^2 \equiv c \pmod q\}.
$$

We also put $A_0(q,c) = A(q,q;q,c)$ and 
define
$$
\Delta(M,N;q,c) = \left|A(M,N;q,c) -
 \frac{MN}{q^2} A_0(q,c)\right|.
$$

It has been shown by Heath-Brown~\cite[Lemma~3]{HB} that 
the bound
\begin{equation}
\label{eq:HB bound}
\sum_{c=1}^q \Delta(N,N;q,c) ^2 \le q^{4/3 + o(1)} r^3,
\end{equation}
holds for $N \le q^{2/3}$, where 
$$
r = \prod_{p=2~\text{or}~\alpha_p>1} p^{\alpha_p}
$$
and 
$$
q= \prod_{p\mid q} p^{\alpha_p}
$$
is the prime number factorisation of $q$.  This estimate is
a part of the suggested in~\cite{HB}  approach to 
the pair correlation problem for the fractional 
parts of the quadratic function $\alpha k^2$, $k=1,2,\ldots$ 
with a real $\alpha$. 

Here we use a different method that leads to an estimate
which improves and generalises~\eqref{eq:HB bound} 
for most of the values of the parameters $M$ and $N$. 
However, in the case of $M, N = q^{2/3 + o(1)}$,
which appears in the  applications pair correlation problem
both bounds are of essentially the same type (except for the
extra factor of $r^3$ in~\eqref{eq:HB bound}, which, however, 
is small for a ``typical'' $q$).

On the other hand, studying the  distribution of solutions to 
the congruence $m^2-n^2 \equiv c \pmod q$, in particular, estimating 
$\Delta(M,N;q,c)$ individually and on average, is of independent interest.

Since there does not seem to be any immediate implications of our estimate
for the pair correlation problem, we present it only in the case of odd $q$.
For even $q$, one can easily obtain a similar result at the cost of some
minor technical changes.

\begin{theorem}
\label{thm:Quad Cong Aver} For any odd $q\ge 1$ and  positive integers $M,N\le q$, 
we have
$$
\sum_{c=1}^q \Delta(M,N;q,c) ^2
\le (M+N)^2  q^{o(1)}.
$$
\end{theorem}

\section{Preliminaries}
\label{sec:prelim} 

As usual, we use $\varphi(k)$ to denote the Euler function
and $\tau(k)$ to denote the divisor function.

\begin{lemma}
\label{lem:A0}
If $q$ is odd and $\gcd(c,q)=d$ then 
$$
A_0(q,c) =  \sum_{f\mid d} f \varphi(q/f).
$$
\end{lemma}

\begin{proof}
As in~\cite[Section~3]{HB} we note that if  an odd $q$ then $A_0(q,c)$
is equal to the number of solutions to the congruence
$$
uv \equiv c \pmod q, \qquad 1 \le u,v \le q.
$$
Now, for every divisor $f\mid d$ we 
collect  together the solutions $(u,v)$ with $\gcd(u,q)=f$. 
Writing $u = fw$ with $1\le w \le q/f$ and $\gcd(w,q/f)=1$,  
we see that $uw \equiv c/f \pmod {q/f}$. 
Thus, for each of the $\varphi(q/f)$ possible values for $w$, the 
corresponding value of 
$u$ is uniquely defined modulo $q/f$ and thus $u$ takes $f$ distinct values
in the range $1 \le u \le q$.
\end{proof}

We also need the following well-known consequence of the sieve of
Eratosthenes.

\begin{lemma}
\label{lem:erat}
For any real numbers $W$ and $Z \ge 1$ and an integer $s \ge 1$, we have
$$
 \sum_{\substack{W < k \le W+Z \\ \gcd(k, s)=1}} 1
= \frac{\varphi (s) }{s}Z + O(\tau(s)).
$$
\end{lemma}

\begin{proof} Using the  from the inclusion-exclusion principle 
we write 
$$
\sum_{\substack{W < k \le W+Z \\ \gcd(k, s)=1}} 1
= \sum_{d \mid s}\mu (d)\sum_{\substack{W < k \le W+Z \\ d \mid k}} 1
$$
where $\mu (d)$ is the 
M\"obius function, see~\cite[Section~16.3]{HardyWright}. 
Therefore, 
\begin{eqnarray*}
\sum_{\substack{W < k \le W+Z \\ \gcd(k, s)=1}} 1
= \sum_{d \mid s}\mu (d)\(Z/d + O(1)\)  
= Z \sum_{d \mid s} \frac{\mu (d)}{d}  + O(\tau(s)).
\end{eqnarray*}
Recalling that
$$
\sum_{d \mid s} \frac{\mu (d)}{d}  = \frac{\varphi (s) }{s}
$$
see~\cite[Equation~(16.3.1)]{HardyWright}, we obtain the 
desired result. 
\end{proof}

Using partial summation, we derive from Lemma~\ref{lem:erat}: 

\begin{cor}
\label{cor:erat}
For any real numbers $W$ and $Z \ge 1$ and an integer $s \ge 1$, we have
$$
 \sum_{\substack{W < k \le W+Z \\ \gcd(k, s)=1}} k
= \frac{\varphi (s) }{2s}Z(W+Z) + O((W+Z)\tau(s)).
$$
\end{cor}

Finally, we recall the bound
\begin{equation}
\label{eq:tau}
\tau(k) = k^{o(1)}, 
\end{equation}
see~\cite[Theorem~317]{HardyWright}.

\section{Products in residue classes}
\label{sec:prod} 

Here we present our main technical tool. 
Assume that for an integer $s$   we are given two sequences
of nonnegative real numbers 
$$\cY = \{Y_u\}_{u=1}^s \mand \cZ = \{Z_u\}_{u=1}^s.
$$

We denote by 
$T(X, \cY, \cZ;s,a)$ the number of solutions to the congruence
$$
u v \equiv a \pmod {s}, \qquad 2 \le u \le X, \ 
\gcd(u,s) =1, \ Z_{u} \le  v \le  Z_{u} + Y_u . 
$$
The following result is an immediate generalisation of~\cite[Theorem~1]{Shp},
which  corresponds to the constant  values of the form $Y_u = Y$ and $Z_u = Z+1$
for some integers $Y$ and $Z$.

\begin{lemma}
\label{lem:prod}
Assume that 
$$
\max_{2 \le u \le X} Y_u  = Y.
$$
Then 
$$
\sum_{a=1}^s
\left|T(X, \cY, \cZ;s,a) - 
\frac{1}{s}  \sum_{\substack{2 \le u \le X\\ \gcd(u,s)=1}} Y_u
\right|^2  \le X(X+Y)s^{o(1)}.
$$
\end{lemma}

\begin{proof}
We recall that by~\cite[Bound~(8.6)]{IwKow}, for $2 \le u \le X$ we have 

$$
   \sum_{Z_{u} \le  v \le  Z_{u} + Y_u} \es(ry) \ll \min\{Y_u, s/|r|\} \ll  \min\{Y, s/|r|\}, 
$$
which holds for any integer   with $0 < |r| \le s/2$.
Now the proof of~\cite[Theorem~1]{Shp} extends to this 
more general case without any changes.
\end{proof}

\section{Proof of Theorem~\ref{thm:Quad Cong Aver}}

Without loss of generality we may assume that 
\begin{equation}
\label{eq:MN}
M\ge N.
\end{equation}

Using the variables
$x = m+n$ and $y = m-n$ we see that $A(M,N;q,c)$
is equal to the the number of solutions to the congruence
\begin{equation}
\label{eq:xy cong}
xy \equiv c \pmod q, 
\end{equation}
where 
\begin{equation}
\label{eq:xy cond}
2 \le x+y \le 2M, \qquad 2 \le x-y \le 2N, \qquad  y \equiv x \pmod 2. 
\end{equation}

Putting $\vartheta_x = 0$ if $x \equiv 0 \pmod 2$ and $\vartheta_x = 1$,
otherwsie, and writing $y = \vartheta_x + 2v$, we see that~\eqref{eq:xy cong}
and~\eqref{eq:xy cond} are equivalent to 
\begin{equation}
\label{eq:xv cong}
x(\vartheta_x + 2v) \equiv c \pmod q, \qquad 2 \le x \le X,  \ L_x \le v \le U_x, 
\end{equation}
where $X = M+N$ and 
\begin{equation}
\begin{split}
\label{eq:LandU}
 L_x & =  \max\left\{1 - \frac{x+\vartheta_x}{2}, 
\frac{x-\vartheta_x}{2} - N\right\}, \\
U_x & =  \min\left\{1+\frac{x-\vartheta_x}{2}, M- \frac{x+\vartheta_x}{2}\right\}.
\end{split}
\end{equation}

We note that it is enough to prove that for every $d \mid q$ we have
\begin{equation} 
\label{eq:gcd=d}
\sum_{\substack{c=1 \\ \gcd(c,q) =d}}^q 
\Delta(M,N;q,c) ^2
\le  M^2q^{o(1)}.
\end{equation}

Now, assume that  $\gcd(c,q)=d$.

For every divisor $f\mid d$, we collect together the solutions to~\eqref{eq:xv cong}
with $\gcd(x,q)=f$ and denote the number of such solutions by $B(M,N;q,c,f)$. 

In particular, if $\gcd(c,q)=d$ then we have
$$
A(M,N;q,c) = \sum_{f \mid d} B(M,N;q,c,f).
$$
Hence, using Lemma~\ref{lem:A0}, the Cauchy inequality and the
bound~\eqref{eq:tau}, we derive
\begin{equation}
\label{eq:Delta and B}
\Delta(M,N;q,c) ^2 \le q^{o(1)} \sum_{f \mid d} 
\left|B(M,N;q,c,f) - \frac{MN f \varphi(q/f)}{q^2}\right|^2. 
\end{equation}

To estimate $B(M,N;q,c,f)$,  writing $x = fu$ with $\gcd(u, q/f)=1$, 
and taking into accoount that since $q$ is odd, we 
have $\vartheta_x = \vartheta_u$, we see that
$B(M,N;q,c,f)$ is equal to the number of solutions to the congruence
\begin{equation}
\label{eq:uv cong}
u(\vartheta_u + 2v) \equiv c_f \pmod {q_f}, 
\end{equation}
where 
$$
 2 \le u \le X_f, \qquad
\gcd(u, q_f)=1, \qquad L_{fu} \le v \le U_{fu}, 
$$
and 
$$
c_f = c/f, \qquad q_f = q/f, \qquad X_f = \fl{X/f}.
$$

We now rewrite~\eqref{eq:uv cong} as $u(2^{-1}\vartheta_u + v) \equiv 2^{-1}c_f \pmod {q_f}$. 
Defining $h_{f,u}$ by the conditions 
$$
2h_{f,u} \equiv \vartheta_u  \pmod {q_f}, \qquad 0 \le h_{f,u} < q_f, 
$$
we see that 
\begin{equation}
\label{eq:BandT}
B(M,N;q,c,f) = T(X_f, \cY_f, \cZ_f;q_f,2^{-1}c_f),
\end{equation}
where $T(X, \cY, \cZ;s,a)$ is defined in Section~\ref{sec:prod} 
and with the  sequences $\cY_f = \{Y_{f,u}\}_{u=1}^{q_f}$ and 
$\cZ_f = \{Z_{f,u}\}_{u=1}^{q_f}$
given by 
$$
Z_{f,u} = h_{f,u} + L_{fu} \mand Y_{f,u} =  U_{fu} - L_{fu}.
$$
In order to apply Lemma~\ref{lem:prod} we need to evaluate
the main term
$$
W_f =  \frac{1}{q_f}\sum_{\substack{u =2\\ \gcd(u,q_f)=1}}^{X_f} 
(U_{fu} - L_{fu}).
$$

Recalling the condition~\eqref{eq:MN} and the definition~\eqref{eq:LandU}, we see that 
$$
U_{fu} - L_{fu} = 
\left\{\begin{array}{ll}
fu + O(1),&\quad\text{if $u \le N_f$,}\\
N+ O(1),&\quad\text{if $N_f < u \le M_f$,}\\
N+M-fu+ O(1),&\quad\text{if $ M_f< u \le X_f $,}
\end{array}
\right.
$$
where 
$$
M_f = \rf{M/f} \mand N_f = \rf{N/f}.
$$
Thus, using Lemma~\ref{lem:erat} and Corollary~\ref{cor:erat}, we derive
\begin{eqnarray*}
W_f &=&   \frac{f}{q_f}\sum_{\substack{u \le N_f\\ \gcd(u,q_f)=1}} u 
+ \frac{N}{q_f} \sum_{\substack{N_f < u \le M_f\\ \gcd(u,q_f)=1}} 1\\
& &\quad +~\frac{M+N}{q_f} \sum_{\substack{N_f < u \le M_f\\ \gcd(u,q_f)=1}} 1 
- \frac{f}{q_f} \sum_{\substack{M_f< u \le X_f\\ \gcd(u,q_f)=1}}u 
+ O(X_fq_f^{-1})\\
 &=& \frac{f\varphi(q_f)}{2q_f^2}N_f^2 
+ \frac{N\varphi(q_f)}{q_f^2} (M_f-N_f)\\
& &\quad +~\frac{(M+N)\varphi(q_f)}{q_f^2} (X_f-M_f)
- \frac{f\varphi(q_f)}{2q_f^2}  (X_f^2-M_f^2) \\
& &\qquad\qquad\qquad\qquad\qquad\qquad\qquad\qquad\qquad
+~O(X_fq_f^{-1}\tau(q_f)).
\end{eqnarray*}
Thus recalling the values of $q_f$ $M_f$, $N_f$ and $X_f$, 
the assumption~\eqref{eq:MN} and using~\eqref{eq:tau},
we see that 
\begin{eqnarray*}
W_f &=&  \frac{fN^2\varphi(q_f)}{2q^2}
+ \frac{fN(M-N)\varphi(q_f)}{q^2}\\
& &\quad +~\frac{fN(M-N)\varphi(q_f)}{q^2}
- \frac{fN(2M-N)\varphi(q_f)}{2q^2}
+ O(Mq^{-1}\tau(q))\\
 &=&  \frac{fMN\varphi(q_f)}{q^2} + O(Mq^{-1+o(1)}).
\end{eqnarray*}
Thus, by the Cauchy inequality and we have
\begin{eqnarray*}
\lefteqn{
\left|B(M,N;q,c,f) - \frac{MN f \varphi(q/f)}{q^2}\right|}\\
& & \qquad \qquad \qquad \le |B(M,N;q,c,f) - W_f|^2 + O(M^2q^{-2+o(1)}).
\end{eqnarray*}
Therefore, we derive from~\eqref{eq:Delta and B} that
$$
\Delta(M,N;q,c) ^2 \le q^{o(1)} \sum_{f \mid d} 
\left|B(M,N;q,c,f) -W_f\right|^2 +  O(M^2q^{-2+o(1)}).  
$$
Hence, 
\begin{eqnarray*}
\lefteqn{\sum_{\substack{c=1 \\ \gcd(c,q) =d}}^q 
\Delta(M,N;q,c) ^2}\\
&  &\qquad \le  \sum_{\substack{c=1 \\ \gcd(c,q) =d}}^q
\sum_{f\mid d}   |B(M,N;q,c,f) - W_f|^2 + O(M^2q^{-1+o(1)}) \\
&  &\qquad \le 
\sum_{f\mid d}   \sum_{\substack{c=1 \\ f \mid c}}^q |B(M,N;q,c,f) - W_f|^2 + O(M^2q^{-1+o(1)}) \\
&  &\qquad \le 
\sum_{f\mid d}   \sum_{\substack{c_f=1}}^{q_f} |B(M,N;q,fc_f,f) - W_f|^2 + O(M^2q^{-1+o(1)}) .
\end{eqnarray*}
Recalling~\eqref{eq:BandT} and applying
Lemma~\ref{lem:prod},  we obtain~\eqref{eq:gcd=d} and
conclude the proof. 

\section*{Acknowledgement}

The author is grateful to Roger Heath-Brown for 
very useful discussions. 
During the preparation of this paper, the author
was supported in part by ARC grant DP1092835.

\end{document}